\documentclass[12pt,reqno]{amsart}
\usepackage{amssymb,amsmath,dsfont}
\usepackage{pgfplots}
\pgfplotsset{compat=1.14}
\usepackage{enumitem}
\usepackage{calc}

\title[Noise sensitivity for minimal left-right crossing of a square]{Noise sensitivity and variance lower bound for minimal left-right crossing of a square in first-passage percolation}

\date{}

\author{Dor Elboim}
\address{Dor Elboim\hfill\break
    Department of Mathematics,
    Stanford University,
    California, United States.}
\email{dorelboim@gmail.com}
\author{Barbara Dembin}
\address{Barbara Dembin\hfill\break
    CNRS, Université de Strasbourg, France}
\email{barbara.dembin@math.unistra.fr}

\usepackage[margin=1in]{geometry}
\usepackage[all]{xy}
\usepackage{amssymb}
\usepackage{amsfonts}
\usepackage{amsthm}
\usepackage{amsmath}
\usepackage{hyperref}
\usepackage{mathtools}
\usepackage{url}
\usepackage{xcolor}
\usepackage{dsfont}
\usepackage{pgfplots}
\usepackage{comment}

\newtheorem{thm}{Theorem}[section]
\newcommand{\thistheoremname}{}
\newtheorem*{genericthm*}{\thistheoremname}
\newenvironment{namedthm*}[1]
  {\renewcommand{\thistheoremname}{#1}%
   \begin{genericthm*}}
  {\end{genericthm*}}
\newtheorem{lem}[thm]{Lemma}  
\newtheorem{prop}[thm]{Proposition}
\newtheorem{cor}[thm]{Corollary}

\newtheorem{definition}[thm]{Definition}
 
\newtheorem{remark}[thm]{Remark}
\newtheorem{claim}[thm]{Claim}

\renewcommand{\P}{\mathbb{P}}

\newcommand{\ZZ}{\mathbb{Z}}
\newcommand{\Z}{\mathbb{Z}}

\newcommand{\EE}{\mathbb{E}}
\newcommand{\E}{\mathbb{E}}

\newcommand{\var}{\mathrm{Var}}
\newcommand{\Inf}{\mathrm{Inf}}

\newcommand{\cP}{\mathcal{P}}

\newcommand{\ep}{\epsilon}
\mathtoolsset{showonlyrefs}
\DeclareMathOperator{\Bin}{Bin}

\numberwithin{equation}{section}

\begin{document}
\begin{abstract}
    We study first-passage percolation on $\mathbb Z ^2$ with independent and identically distributed weights, whose common distribution is uniform on $\{a,b\}$ with $0<a<b<\infty $.
    
    Following  Ahlberg and De la Riva \cite{ahlberg2023being}, we consider the passage time $\tau (n,k)$ of the minimal left-right crossing of the square $[0,n]^2$, whose vertical fluctuations are bounded by~$k$. We prove that when $k\le n^{1/2-\epsilon}$, the event that $\tau (n,k)$  is larger than its median is noise sensitive. This improves the main result of \cite{ahlberg2023being} which holds when $k\le n^{1/22-\epsilon }$. 

    Under the additional assumption that the limit shape is not a polygon with a small number of sides, we extend the result to all $k\le n^{1-\epsilon }$. This result follows unconditionally when $a$ and $b$ are sufficiently close.
    
    Under a stronger curvature assumption, we extend the result to all $k\le n$. This in particular captures the noise sensitivity of the event that the minimal left-right crossing $T_n=\tau (n,n)$ is larger than its median.
    
    Finally, under the curvature assumption, our methods give a lower bound of $n^{1/4-\epsilon }$ for the variance of the passage time $T_n$ of the minimal left-right crossing of the square. We prove the last bound also for absolutely continuous weight distributions, generalizing a result of Damron--Houdré--Özdemir \cite{Damron2024}, which holds only for the exponential distribution.

   Our approach differs from \cite{ahlberg2023being,Damron2024}; the key idea is to establish a small ball probability estimate in the tail by perturbing the weights for tail events using a Mermin--Wagner type estimate.
\end{abstract}

\maketitle

\section{Introduction}
First-passage percolation is a model for a random metric space, formed by a random perturbation of an underlying base space. Since its introduction by Hammersley--Welsh in 1965~\cite{HammersleyWelsh}, it has been studied extensively in the probability and statistical physics literature. We refer to \cite{Kesten:StFlour} for general background and to \cite{50years} for more recent results.


We study first-passage percolation on the square lattice $(\ZZ^2,E(\mathbb{Z}^2))$, in an independent and identically distributed (IID) random environment. The model is specified by a \emph{weight distribution $G$}, which is a probability measure on the non-negative reals. It is defined by assigning each edge $e\in E(\mathbb{Z}^2)$ a random passage time $t_e$ with distribution $G$, independently between edges. Then, each finite path $p$ in $\mathbb Z ^2$ is assigned the random passage time
\begin{equation}\label{eq:time of a path}
    T(p):=\sum _{e\in p} t_e,
\end{equation}
yielding a random metric $T$ on $\ZZ^2$ by setting the \emph{passage time between $u,v\in \mathbb Z ^2$} to
\begin{equation}\label{eq:defT}
    T(u,v):=\inf_{p} T(p),
\end{equation}
where the infimum ranges over all finite paths connecting $u$ and $v$. Any path achieving the infimum is termed a \emph{geodesic} between $u$ and $v$. 
The focus of first-passage percolation is the study of the large-scale properties of the random metric $T$ and its geodesics.

\subsection{The time constant and the limit shape}

When $G$ has a finite moment, one can prove a \emph{law of large numbers}: asymptotically when $n$ is large, the random variable $T(0,nx)$ behaves like $n \cdot \mu(x)$ where $\mu(x)$ is a deterministic constant depending only on the distribution $G$ and the point $x\in\Z^2$. More precisely, for every $x\in\Z^2$, there exists a deterministic constant $\mu(x)$ such that
\[\lim_{n\rightarrow\infty}\frac{T(0,nx)}{n}=\mu(x)\qquad\text{almost surely and in $L^1$}.\] 
The limit $\mu(x)$ is the so-called \emph{time constant} in the direction $x$. It is well known (see, e.g., \cite{50years}) that the function $\mu $ can be extended to all $x\in \mathbb R^2$ and is a norm on $\mathbb R^2$.

The unit ball $\mathcal B _G$ of the norm $\mu $ is called the \emph{limit shape} corresponding to the distribution $G$, and it controls the shape of a large ball in the metric $T$. This is shown in the classical limit shape theorem by Cox and Durrett~\cite{CoxDurrett}. To state the theorem, define the ball of radius $t$ by 
\begin{equation}
B(t):=\big\{\,v\in\ZZ^2:\, T(0,v)\leq t\,\big\}+ \Big[ -\frac{1}{2},\frac{1}{2}\Big] ^2.
\end{equation}

\begin{namedthm*}{Limit Shape Theorem}[Cox and Durrett~\cite{CoxDurrett}] 
For any integrable distribution $G$,
 there exists a deterministic convex set $\mathcal B_G$ such that for all $\epsilon >0$,
\begin{equation}
     \mathbb P \Big( \exists t_0>0  \text{ such that }\  \forall t\geq t_0, \quad (1-\ep)\mathcal B_G \subseteq \frac{B(t)}{t}\subseteq (1+\ep)\mathcal B_G \Big) =1.
\end{equation}
Moreover, we have \[\mathcal B_G\coloneqq\{x\in\mathbb R^2: \mu(x)\le 1\}.\]
\end{namedthm*}











\subsection{Noise sensitivity and the BKS theorem} 

Roughly speaking, a sequence of boolean function is said to be noise sensitive if a small, random perturbation of the inputs, leads to an asymptotically independent output. In this paper we consider functions (or events) of configurations of edge weights in the square $[0,n]^2$. That is, we assume that the edge distribution $G$ is supported in $\{a,b\}$ and consider boolean functions of the form $f_n:\{a,b\}^{E_n}\to \{0,1\}$, where $E_n$ is the set of edges having both endpoints in the square $[0,n]^2$.
Let $t$ be a random variable on $\{a,b\}^{E_n}$ distributed as $G^{\otimes E_n}$. Let $t^\ep$ be the random variable where for each edge $e$ the value $t_e$ is resampled independently with probability $\ep$. 
\begin{definition}The sequence $(f_n)_n$ is noise sensitive if for every $\ep>0$
\[\lim_{n\rightarrow \infty}\EE[f_n(t)f_n(t^\ep)]-\EE[f_n(t)]^2=0.\]
A sequence of events $\mathcal A_n \subseteq \{a,b\}^{E_n}$ is noise sensitive if $(\mathds 1 _{\mathcal A _n})_n$ is noise sensitive.
\end{definition}
For a function $f:\{a,b\}^{E_n}\to \{0,1\}$, we define the influence $\Inf_e(f)$ of an edge $e$ to be the probability that changing the weight of the edge $e$ modifies the outcome of $f$. More precisely, letting $\sigma_e:\{a,b\}^{E_n}\to \{a,b\}^{E_n}$ be the function that switches the weight of the edge $e$, we have
\[\Inf_e(f)\coloneqq\P(f(t)\ne f(\sigma_e(t)).\]

The following theorem due to Benjamini--Kalai--Schramm~\cite[Theorem~1.3]{BKS2} gives a useful criterion for noise sensitivity of a function in terms of the influences.

\begin{namedthm*}{BKS Theorem}[Benjamini--Kalai--Schramm~\cite{BKS2}]
Let $f_n:\{a,b\}^{E_n}\to \{0,1\}$ be a sequence of boolean functions such that  
\[\lim_{n\rightarrow \infty}\sum_e\Inf_e(f_n)^2=0.\]
Then, the sequence $(f_n)_n$ is noise-sensitive.
\end{namedthm*}
This condition is in fact a necessary condition when for all $n\ge1$ $f_n$ is monotone with respect to the partial order on $\{a,b\}^{E_n}$ (see \cite[Theorem~1.4]{BKS2}).
\subsection{Main results}

Denote by $T_n$ the minimal passage time of a path connecting the left and right boundaries of the square $[0,n]^2$. 
Define $\cP_k(n)$ be the set of paths from left to right in the square $[0,n]^2$ such that the vertical displacement is at most $k$.
We define $\tau(n,k)$ to be the minimal passage time between the left and the right of the square $[0,n]^2$ for paths with vertical displacement at most $k$, that is 
\begin{equation*}
    \tau(n,k)\coloneqq \inf\{T(\gamma): \gamma\in\cP_k(n)\}
\end{equation*}
It is believed that $T_n=\tau(n,n^{2/3+\ep})$ with high probability. This motivates the study of the restricted passage time $\tau(n,k)$.

In most of the paper, we will consider a two point distribution $G$. Namely, we will assume
\begin{equation}\label{eq:assumption atomic}
    \quad G\text{ is supported on $\{a,b\}$ where $0< a<b<\infty $}. \tag{ATO}
\end{equation}
For some of our results we will need to have a finer control on the geometry of the geodesic. This requires a better control on the tail of $G$ and on the limit shape.
The first assumption says that the limit shape is not a polygon with small number of sides.
\begin{equation}\label{ass:sides}
\begin{split}
      &\text{The limit shape } \mathcal B _G \text{ is either not a polygon,}\\
      & \quad \quad \quad \text{or a polygon with at least } s \text{ sides.}
\end{split}
 \tag{$\ge s$ sides}
\end{equation}
The second assumption says that the limit shape has positive curvature in  direction $\mathrm e_1$.
\begin{equation}\label{ass:uniform curbature}
\begin{split}
      \text{There is a constant }& c>0 \text{ such that for every $|h|<1$} \\ & \mu(\mathrm e_1+h\mathrm e_2)-\mu(\mathrm e_1)\ge c h^2.
\end{split}
 \tag{UC}
\end{equation}
These two assumptions are expected to hold for almost any distribution. However, the first assumption is known only for some weight distributions while the second assumption is not known for a single distribution. See discussion in Section~~\ref{sec:remarks} below.


In our first result, we study the noise sensitivity of the passage times $\tau(n,k)$ and $T_n$. To state the theorem, for any $\alpha\in(0,1)$, we denote by  $q_\alpha(X)$ the $\alpha $-quantile of the random variable $X$. Namely, we let  $q_\alpha (X):=\sup \{ x\in \mathbb R : \mathbb P (X\le x)\le \alpha \}$.

\begin{thm}\label{thm:noisesensitivity}
    Suppose that $G$ satisfies \eqref{eq:assumption atomic}. Fix $\alpha \in (0,1)$ and for a sequence of integers $(k_n)_{n\ge 1}$, consider the events 
    \begin{equation}
        \mathcal A _n(\alpha ,k_n):= \big\{\tau (n,k_n)\ge q_\alpha (\tau (k_n,n))\big\}.
    \end{equation}
    The following holds:
    \begin{enumerate}
        \item There exists $C>0$ depending only on $G$ and $\alpha $, such that if $k_n\le e^{-C\sqrt{\log n}}n^{1/2}$ for all $n\ge 1$ then the sequence $( \mathcal A _n(\alpha ,k_n))_{n\ge 1}$ is noise sensitive.
        \item For any $\epsilon >0$ there exists $s>0$ depending only on $\epsilon $ such that if \eqref{ass:sides} holds and $k_n\le n^{1-\epsilon }$ for all $n\ge 1$, then the sequence $(\mathcal A _n(\alpha ,k_n))_{n\ge 1}$ is noise sensitive.
        \item Suppose that \eqref{ass:sides} holds for a sufficiently large $s$ and that \eqref{ass:uniform curbature} holds. If $k_n\le n$ for all $n$ then the sequence $(\mathcal A _n(\alpha ,k_n))_{n\ge 1}$ is noise sensitive. In particular, the sequence of events $\{T_n\ge q_\alpha (T_n)\}$ is noise sensitive. 
    \end{enumerate}
\end{thm}

Our methods can be used to obtain variance lower bounds for the passage times $\tau (n,k)$ and $T_n$. We prove these bounds for a wider class of weight distributions. Specifically, we either assume that distribution $G$ satisfies \eqref{eq:assumption atomic} or that it is integrable and absolutely continuous. That is, in the latter case we assume
\begin{equation}\label{eq:assumption abs1} 
    G\text{ is integrable and absolutely continuous with respect to Lebesgue measure.} 
 \tag{ABS}
\end{equation}
and
\begin{equation}\label{eq:assumption exp1}
\begin{split}
     G\text{ has an exponential moment. Namely, we have  $\mathbb E [e^{\alpha t_e}]<\infty $ for some $\alpha >0$,}
\end{split}
 \tag{EXP}
\end{equation}
where $t_e\sim G$.

      

\begin{thm} \label{thm:lowerboundvar}Suppose that $G$ either satisfies \eqref{eq:assumption abs1} or \eqref{eq:assumption atomic}. Then:
\begin{enumerate}
    \item 
    There exists $C>0$ such that for all $n\ge 1$ and $k\le n$ 
\begin{equation*}
   \var(\tau(n,k))\ge e^{-C\sqrt{\log \frac nk }}\left(\frac nk\right).
\end{equation*}
    \item
    Suppose in addition that $G$ satisfies \eqref{eq:assumption exp1} and that \eqref{ass:uniform curbature}  holds. Then, for any $\epsilon >0$ there exists $c>0$ such that for all $n\ge 1$
\begin{equation*}
   \var(T_n)\ge cn^{1/4-\ep}.
\end{equation*}
\end{enumerate}
\end{thm}


    

\subsection{Remarks and extensions}\label{sec:remarks}
     Assumption~\eqref{ass:sides} is very weak and expected to hold for any weight distribution $G$. In fact, for some weight distributions it is possible to prove it. In our previous paper \cite[Theorem~1.5]{dembin2024coalescence} it is shown that if the weight distribution is a small perturbation of a constant distribution, then \eqref{ass:sides} holds. Applying this theorem for a distribution satisfying \eqref{eq:assumption atomic} we obtain the following. For any $p\in (0,1)$ and any integer $s$, there exists $\delta >0$ such that for any $0<a<b<\infty $ with $(b-a)/a<\delta $, the distribution $G$ supported on $\{a,b\}$ with $G(\{a\})=p$ satisfies assumption \eqref{ass:sides}. Hence, in this case part (2) of Theorem~\ref{thm:noisesensitivity} holds unconditionally.

In the proof of Theorem~\ref{thm:noisesensitivity}, Assumption~\eqref{ass:sides} is only used to prove that typically a geodesic does not intersect a given vertical line in many places. If one knew that with high probability the geodesic does not intersect more than constant time any given vertical line then one would get the result without limit shape assumption for any $k_n\le e^{-C\sqrt {\log n }} n$. In particular, part (2) of Theorem~\ref{thm:noisesensitivity} holds unconditionally in any directed version of first-passage percolation.

For first-passage percolation on $\mathbb Z ^2$, Assumption~\eqref{ass:uniform curbature} is not known for a single weight distribution $G$. However, there are other models of first-passage percolation in which the limit shape can be computed exactly. For such models Assumptions~\eqref{ass:uniform curbature} and~\eqref{ass:sides} are easily verified. Examples of such models are the rotationally symmetric models of first-passage percolation (see, e.g., \cite{basu2023rotationally}), in which the limit shape is a ball, and harmonic first-passage percolation \cite{dembin2024minimal}, in which the limit shape is a parabola. We expect our methods to extend to these models (with an appropriate definition of noise sensitivity), to obtain an unconditional proof of the noise sensitivity of the left-right crossing passage time being large than its median.

    Our proofs are not specific to dimension 2, the same proofs will work for higher dimensions and give
    \begin{equation*}
   \var(\tau(n,k))\ge e^{-C\sqrt{\log \frac nk }}\left(\frac n{k^{d-1}}\right).
\end{equation*}

The noise sensitivity question in higher dimensions becomes straightforward. One expects that the probability that the geodesic for the left-right crossing passes through an edge is of order $n^{-(d-1)}$. Bounding the influences by the probability that the edge lies on the geodesic is enough to show that the sum of the squared influences tends to 0, even when disregarding the intersection with the event where the passage time is close to its median.

\subsection{Comparison with previous works}
Ahlberg--De la Riva \cite{ahlberg2023being} first proved Theorem \ref{thm:noisesensitivity} for $k_n \le n^{1/22}$. Their result was the motivation for this work. The key element of their proof is obtaining a good upper bound for the probability that $\mathcal T(n,k)$, the passage time of the minimal left-right crossing in the rectangle $[0,n] \times [0,k]$, lies within the interval $[q_\alpha(\tau(n,k)) - b + a, q_\alpha(\tau(n,k))]$. To bound this probability, Ahlberg--De la Riva derived precise estimates for the moderate deviations of the passage time in thin rectangles (see \cite[Section 4]{ahlberg2023being}) by computing improved Cramér-type bounds. This approach imposes a restriction on the value of $k$, as the rectangles must be sufficiently thin to obtain good estimates on moderate deviations, allowing for small ball probability bounds. We take a different approach to upper-bound this quantity that does not require to obtain precise estimates on moderate deviations. That allows us to obtain a stronger result with a weaker condition on $k$ (see Theorem \ref{thm:noisesensitivity}). This allows for a conditional resolution of the original question posed by Benjamini, which motivated the work \cite{ahlberg2023being}:``Is being above the median for $T_n$ noise sensitive?".

In Damron--Houdré--Özdemir \cite{Damron2024}, the authors focus on the exponential distribution and crucially use the memoryless property of exponential distributions to establish a lower bound on the variance of order $n^{1/4-\ep }$ under the assumption \eqref{ass:uniform curbature}. The second part of Theorem~\ref{thm:lowerboundvar} (which follows from the unconditional, first part) extends their result to distributions that satisfy \eqref{eq:assumption abs1}+\eqref{eq:assumption exp1} or \eqref{eq:assumption atomic}.

A recent work by the second author \cite{elboim2024small}, establishes a small ball probability bound for the point to point passage time. More precisely, it is shown in \cite{elboim2024small} that as $\|x\|\to \infty $ one has  $\sup _{a>0}\mathbb P (T(0,x)\in [a,a+1] )\to 0$. This is done for the class of absolutely continuous weight distribution and answers a question raised in the paper of Ahlberg--De la Riva \cite{ahlberg2023being}. In the proof of Proposition~\ref{prop:MW} below, we use some of the ideas of \cite{elboim2024small} in order to obtain a local probability bound for an interval $[a,a+1]$ that is far in the tail of the passage time distribution.

\subsection{Sketch of proof} 
Let $n\ge k\ge 1$. For short, write $q_\alpha $ for $q_\alpha(\tau(n,k))$.
Uniformly perturbing the weights of edges in $[0,n]^2$ only results in a constant lower bound for the variance, while perturbing the weights of edges in $[0,n] \times [0,k]$ has a more significant impact on $\mathcal T(n,k)$, the passage time of the minimal left-right crossing in the rectangle $[0,n] \times [0,k]$. To explain the main ideas of the proof, let us first make the following approximation:

\[
\tau(n,k) = \min_{1 \le i \le n/k} \mathcal T_i,
\]
where $(\mathcal T_i)_i$ are independent copies distributed as $\mathcal T(n,k)$. This approximation is quite close to the actual situation in our applications. This observation implies that $\mathcal T(n,k) \le q_\alpha $ is a tail event. We then perturb the weights in the tail of $\mathcal T(n,k)$ to estimate the probability that $\mathcal T(n,k)$ lies within the interval $[q_\alpha  - b + a, q_\alpha ]$. Once we establish an upper bound for this probability, we can derive an upper bound for the probability of $\tau(n,k)$ lying within $[q_\alpha  - b + a, q_\alpha ]$ by using a union bound:

\[
\P(\tau(n,k) \in [q_\alpha  - b + a, q_\alpha ]) \le \frac{n}{k} \P(\mathcal T(n,k) \in [q_\alpha  - b + a, q_\alpha ]).
\]

We then estimate the small ball probability on the right-hand side by perturbing the weights for tail events of $\mathcal T(n,k)$ (see Proposition \ref{prop:MW}). In fact, we prove a control for $0 \le x \le q_\alpha$ of
\(\P(\tau(n,k) \in [x, x+1]).
\)

This bound is used to establish a lower bound on the variance. The idea is that the probability mass cannot stay concentrated in too small interval to the left of $q_\alpha$. This implies that there must be some positive mass further left from $q_\alpha$, and since by definition there is positive mass to the right of $q_\alpha$, this leads to a lower bound on the variance.

To bound the influence of edges, we first work with a more symmetric version of $\tau(n,k)$ by considering the same variable when gluing the top and bottom of the square. This more symmetric version agrees with $\tau(n,k)$ with high probability, so we can deduce the noise sensitivity of $\tau(n,k)$ from the noise sensitivity of the modified version. This modified version has the advantage of having vertical symmetry. We then upper bound the influence of edges by noting that for an edge to be influential, we need $\tau(n,k) \in [q_\alpha  - b + a, q_\alpha ]$ and the edge to be on the geodesic. To obtain a bound on the influence, we also need to control the expected intersection of the geodesic with a given column (which is done in Section \ref{sec 5}).

To make the approximation rigorous, we work with a slightly modified version of $\mathcal T(n,k)$, denoted $ T(n,k)$. The minimum will be taken over at most $3 \frac{n}{k}$ copies of $ T(n,k)$, which are not necessarily independent. However, it is sufficient to extract a subset of order $\frac{n}{k}$ independent copies  to obtain the desired result (see Lemma \ref{lem:cormw}).

\subsection{Organization of the paper}
In Section \ref{sec 2}, we derive a small ball probability bound in the tail for a slightly different definition of $\mathcal T(n,k)$. In section \ref{sec 3}, we prove Theorem \ref{thm:lowerboundvar}. In section \ref{sec 4}, we prove Theorem \ref{thm:noisesensitivity}. In section \ref{sec 5}, we control the vertical fluctuations of geodesics for $T_n$ under assumption \eqref{ass:uniform curbature} and we bound the intersection of the geodesic for $\tau(n,k)$ with vertical lines.

\subsection{Acknowledgments}
We thank Ron Peled and Eviatar Procaccia for fruitful discussions at an early stage of this project.

\section{Small ball probability bound in the tail}\label{sec 2}

Let $n\ge 4$ and let $1\le k \le n/4$. Let $T(n,k)$ be the passage time of minimal left-right crossing in the rectangle $[0,n]\times [0,2k]$ and vertical displacement at most $k$. In this section, we aim to give an upper bound on the small ball probability for $T(n,k)$ in its tail. Namely, we bound the probability that $T(n,k)\in[a,a+1]$ for $a$ that is far in the tail of the distribution of $T(n,k)$.
\subsection{A Mermin--Wagner type estimate}

The following lemma is the main technical tool required for the proof of the main theorem. The lemma is a Mermin--Wagner type estimate and is taken from \cite[Lemma~2.12 and Remark 2.15]{dembin2024coalescence}. A similar lemma was used also in \cite{dembin2025influence,elboim2024small}.
This lemma enables to perturb the weights and control the probability that a given event $A$ occurs in the perturbed environment compared to the probability that the same event occurs in the original environment. In our applications, we will perturb the environment by decreasing or increasing all the weights inside a cylinder.
\begin{lem}\label{lem:MW}
  Suppose that $G$ is an absolutely continuous distribution on $\mathbb R$. Then, there exist 
  \begin{itemize}
  \item  A Borel set $S_G\subseteq \mathbb R $ with $G(S_G)=1$.
      \item Borel subsets $(B_{\delta })_{\delta>0}$ of $S_G$ with $\lim_{\delta\downarrow 0}G(B_\delta)=1$,
      \item For each $\tau \in \mathbb R$, a bijection $g_{\tau }:S_G\to S_G$.
  \end{itemize}
   such that the following holds:
  \begin{enumerate}
  \item The function $g_\tau (s)$ is increasing both in $s$ and in $\tau $.
      \item $g_0$ is the identity function on $S_G$ and for any $\tau _1,\tau _2\in \mathbb R$ we have $g_{\tau _1}\circ g_{\tau _2}=g_{\tau _1+\tau _2}$.
      \item For any $\tau \in[0,1]$ and $s\in B_\delta $ we have  $g_{\tau }(s)\ge s+\delta \tau $.
      \item For any integer $n\ge 1$, a vector $\tau\in \mathbb R ^n$, any $p,q>1$ with $1/p+1/q=1$ and a Borel set $A\subset\mathbb{R}^n$ we have
      \begin{equation}\label{eq:MW probability estimate}
             \mathbb P (X\in A) \le e^{\frac q{2p}\|\tau \|_2^2} \mathbb P \big(  (g_{q\tau _i/p}(X_i)) _{i=1}^n\in A \big)^{1/q}\mathbb P \big( (g_{-\tau_i}(X_i))_{i=1}^n \in A \big)  ^{1/p}
      \end{equation}
where $X=(X_1,\dots ,X_n)$ is a vector of i.i.d.\ random variables with distribution $G $.
  \end{enumerate}
\end{lem}

\subsection{Small ball probability for continuous distribution}

In this section, we focus on distributions $G$ that satisfy \eqref{eq:assumption abs1}.
The aim of this section is to prove the following control on small ball probability.
\begin{prop}\label{prop:MW} There exist $C,c>0$ depending only on $G$ such that for any $n\ge k\ge 1$, any $a>0$ and $\ep>0$, we have that 
    \begin{equation}
        \mathbb P \big( T(n,k)\in [a,a+1] \big) \le Ce^{C/\ep}\sqrt{\frac k n} \mathbb P \big( T(n,k)\le a+1 \big)^{1-\ep}+e^{-c\sqrt n}.
    \end{equation}
\end{prop}

Denote by $E(n,k)$ the edges with both endpoints in $[0,n]\times[0,k]$.
For $r\in [-2,2]$ we define the perturbation function $\tau _r:E(n,2k)\to [0,\infty )$ by
\begin{equation}
    \forall e\in E(n,2k)\qquad \tau _r (e):= \frac{r}{\sqrt{kn}}, 
\end{equation}
We define the modified environment by \[t_{e,r}:=g_{\tau _r(e)}(t_e).\] We denote by $T_r(p)$ the weight a path $p$ and by $T_r(n,k)$  the passage time of minimal left-right crossing in the rectangle $[0,n]\times [0,2k]$ and vertical displacement at most $k$ in the modified environment $(t_{e,r})_e$ (that is $T_r(n,k)=(T(n,k))(t_{e,r})$). In our proof we continuously vary $r$ and analyze the effect it has on the passage time $T(n,k)$.
The proof of Proposition \ref{prop:MW} will follow from the next lemma which controls the average impact the perturbation has on the small ball probability.
\begin{lem}\label{cor:1}
    There exist $C,c>0$ depending only on $G$ such that for any $a\ge 0$ and $\ep>0$ we have 
    \begin{equation}
        \int _{-1}^1 \mathbb P \big(  T_r(n,k) \in [a,a+1] \big) dr \le Ce^{C/\ep}\P(T(n,k)\le a+1)^{1-\ep}\sqrt{\frac kn}+ e^{-c\sqrt n}.
    \end{equation}
\end{lem}

Let us first explain how to conclude the proof of Proposition~\ref{prop:MW} using Lemma \ref{cor:1}.

\begin{proof}[Proof of Proposition~\ref{prop:MW}]
Note that $\|\tau_r\|_2^2\le C$.
   Thanks to Lemma \ref{lem:MW} (applied to $q=p=2$ and $A=\{T(n,k)\in[a,a+1]\}$), we have for any $r\in [0,1]$ and $a>0$, we have \begin{equation}\label{eq:mw}
        \mathbb P \big( T(n,k)\in [a,a+1] \big) \le C \sqrt{ \mathbb P \big( T_r(n,k)\in [a,a+1] \big) \mathbb P \big( T_{-r}(n,k)\in [a,a+1] \big)}.
    \end{equation}
    Integrating \eqref{eq:mw} over $r\in [0,1]$ and using Cauchy-Schwartz inequality we obtain
    \begin{equation}
        \mathbb P \big( T(n,k)\in [a,a+1] \big) \le C \sqrt{ \int _0^1\mathbb P \big( T_r(n,k)\in [a,a+1] \big) dr \int_0^1 \mathbb P \big( T_{-r}(n,k)\in [a,a+1] \big)dr}.
    \end{equation}
    We finish the proof by bounding the right hand side of the last equation using Lemma~\ref{cor:1}.  
\end{proof}

To prove Lemma~\ref{cor:1}, we first need to 
prove that any path starting on the left side of length at least $n$ is affected by the perturbation. Let $\delta _0$ such that $G(B_{\delta _0})\ge 0.999$ where $B_{\delta_0}$ was defined in Lemma \ref{lem:MW}. This choice imply that $\P(t_{e,r}\ge t_e+\delta_0\tau_r(e))\ge 0.999$. This will ensure that with very high probability, any path from left to right in the cylinder has at least half of its edges significantly affected by the perturbation. This is the purpose of the following lemma.

Denote by $\cP$ the set of paths in $[0,n]\times[0,2k]$ from left to right that have vertical displacement bounded by $k$.

\begin{lem}\label{lem:111}
For any $ r\in [0,1]$ we have that 
\begin{equation}
\mathbb P \Big( \exists p\in \mathcal P, \  T_{r}(p)-T(p)\le \delta _0r\sqrt{\frac n k}  \Big) \le e^{-n}.    
\end{equation}
\end{lem}

\begin{proof}
    Let $r\in [0,1]$. Recall that if $t_e\in B_{\delta _0}$ then 
    \begin{equation}
     t_{e,r}=g_{\tau _r(e)} (t_e)\ge t_e+\delta _0\tau _r(e)=t_e+\frac{\delta _0r}{\sqrt{kn}}.
    \end{equation}
     Thus, letting $\mathcal P'$ be the set of paths which are given by the first $n$ steps of a path in $\mathcal P$, we obtain for a fixed  $p\in \mathcal P'$ that 
    \begin{equation}
    \begin{split}
\mathbb P \Big( T_{r}(p)-T(p)\le  \delta _0r \sqrt{\frac n k} \Big) \le \mathbb P \Big( \big| \big\{e\in p : t_e\in B_{\delta _0} \big\} \big| \le \frac n2 \Big) \le 
\mathbb P \Big( \text{Bin}(n,0.999) \le \frac n 2\Big) \le 8^{-n}.    
\end{split}
\end{equation}
Moreover, we have that $|\mathcal P'| \le k3^{n}$ as we can first choose the stating point of $p\in \mathcal P'$ and then choose each step of the path. The lemma follows from a union bound over $p\in \mathcal P'$ as long as $n$ is sufficiently large.
\end{proof}
The following corollary control the impact of adding a perturbation to an already perturbed environment.
\begin{cor}\label{cor:2}
    For any $s\in [-1,1]$ and $r\in [0,1]$ we have that 
\begin{equation}
\mathbb P \Big( \exists p\in \mathcal P, \  T_{s+r}(p)-T_s(p)\le  \delta _0r\sqrt{\frac n k} \Big) \le Ce^{-n/2}.    
\end{equation}
\end{cor}

\begin{proof}
The corollary follows from Lemma~\ref{lem:111} using Lemma~\ref{lem:MW}. Let us note that this will not be the main use of Lemma~\ref{lem:MW} and that this part can be done using different arguments.
The set $A\subseteq \mathbb R ^{E(n,2k)}$ from Lemma~\ref{lem:MW} will be 
\begin{equation}
    A:=\Big\{ (z_e)_{e\in E(n,2k)} :  \exists p\in \mathcal P  , \ \sum _{e\in p} g _{\tau _{s+r}(e) }(z_e)-g_{\tau _s(e)}(z_e) \le \delta _0r\sqrt{\frac n k}\Big\}.
\end{equation}
By Lemma~\ref{lem:MW} for $p=q=2$ and the fact that $\|\tau _s\|_2\le C$ we have that
\begin{equation}
\mathbb P \big( (t_e)_{e\in E(n,2k)} \in A \big) \le C\sqrt{\mathbb P \big( (g_{-\tau _s(e)}(t_e))_{e\in E(n,2k)} \in A \big) }.    
\end{equation}
It is easy to check (using part (2) of Lemma~\ref{lem:MW}) that the event on the right hand side of the last inequality is precisely the event of Lemma~\ref{lem:111} while the event on the left hand side is the event of the corollary.
\end{proof}

We can now prove Lemma~\ref{cor:1}.

\begin{proof}[Proof of Lemma \ref{cor:1}]
Let $r_0:= \frac 8{\delta_0}\sqrt {\frac k n} $  and let $R:=r_0\mathbb Z \cap [-1,1]$. Define the event
\begin{equation}
    \Omega :=\bigcap _{r\in  R}  \Big\{ \forall p \in \mathcal P, \  T_{r+r_0} (p)-T_{r}(p) > 4 \Big\}.  
\end{equation}
Using Corollary~\ref{cor:2} and a union bound we obtain
\begin{equation}\label{eq:Omega}
\mathbb P (\Omega ^c )\le  C \sum _{r\in R} e^{-n/2}   \le  e^{-c\sqrt{n}}.
\end{equation}
Next, we write 
\begin{equation}\label{eq:1}
        \int _{-1}^1 \mathbb P \big(  T_r(n,k) \in [a,a+1] \big) dr \le \mathbb P (\Omega ^c)+  \int _{-1}^1 \mathbb P \big( \Omega , \ T_r(n,k) \in [a,a+1] \big) dr.
    \end{equation}
Let us bound the second term on the right hand side of \eqref{eq:1}. On $\Omega$, for any path $p$ connecting the left to the right and for all $r\in R$ we have $T_{r+r_0}(p)-T_r(p)\ge 4$. Hence, on $\Omega $, we have $T_{r+r_0}(n,k)-T_r(n,k)\ge 4$ for all $r\in R$. Thus, using also that $T_r(n,k)$ is increasing in $r$ we obtain that on $\Omega$ there is at most one element $r\in R$ for which $T_r(n,k)\in [a,a+1]$ and that \[\big| \big\{r\in [-1,1] : T_r(n,k)\in [a,a+1] \big\}\big| \le r_0\le C\sqrt{\frac k n},\] where in here $|\cdot |$ denotes the Lebesgue measure of the set. Moreover, it is easy to check that since $T_r(n,k)$ is increasing in $r$, we have 
 \begin{equation*}
     \big| \big\{r\in [-1,1] : T_r(n,k)\in [a,a+1] \big\}\big|=\big| \big\{r\in [-1,1] : T_r(n,k)\in [a,a+1] \big\}\big|\mathds{1}_{T_{-1}(n,k)\le a+1}
 \end{equation*}
 Hence, using Fubini's theorem
    \begin{equation}\label{eq:after Fubini}
    \begin{split}
        \int _{-1}^1 \mathbb P \big( \Omega , \ T_r(n,k) \in [a,a+1] \big) dr&= \mathbb E \Big[ \mathds 1 _{\Omega }\big| \big\{r\in [0,1] : T_r(n,k)\in [a,a+1] \big\}\big| \Big] \\&\le C\P(T_{-1}(n,k)\le a+1)\sqrt{\frac kn}.
        \end{split}
    \end{equation}


We now apply Lemma \ref{lem:MW} with the set $A\subseteq \mathbb R ^{E(n,2k)}$ given by
\begin{equation}
    A:=\Big\{ (z_e)_{e\in E(n,2k)} : \exists p\in \mathcal P \text{ with } \ \sum _{e\in p} g_{\tau _{-1}}(z_e) \le a+1\Big\},
\end{equation}
with $p,q$ defined by $1/q=1-\ep$ and $1/p=\ep$, and with the vector $\tau $ being $\sigma =(\sigma _e)_{e\in E(n,2k)}$ defined by $\sigma _e:=\tau_{p/q}(e)$. This gives
\begin{equation}\label{eq:MWT-1}
\begin{split}
    \P(T_{-1}(n,k)\le a+1)&=\mathbb P \big( (t_e)_{e\in E(n,2k)}\in A \big) \le e^{\frac{q}{2p}\|\sigma \|_2^2 }\P\big( (g_{\tau _1}(t_e))_{e\in E(n,2k)}\in A \big) ^{\frac 1 q} \\
    &=e^{\frac{q}{2p}\|\sigma \|_2^2 }\P\big( T(n,k)\le a+1 \big) ^{\frac 1 q} \le e^{C/\ep}\P(T(n,k)\le a+1)^{1-\ep},
\end{split}
\end{equation}
where in the last inequality we used that $\|\sigma \|^2\le C \frac {p^2}{q^2}$. Substituting this estimate into \eqref{eq:after Fubini} finishes the proof of the lemma.
\end{proof}

\subsection{Small ball probability for atomic distribution}
In this section, we focus on distributions $G$ that satisfy \eqref{eq:assumption atomic}. In our previous works \cite{dembin2024coalescence,dembin2025influence,elboim2024small}, the use of Mermin--Wagner (Lemma \ref{lem:MW}) was limited to absolutely continuous distributions. We here explain how to adapt the proof of Proposition \ref{prop:MW} in this context.
We aim to prove the following result.
\begin{prop}\label{prop:MW2} There exist $C,c>0$ depending only on $G$ such that for any $n\ge k\ge 1$, $a>0$ and $\ep>0$, we have that 
    \begin{equation}
        \mathbb P \big( T(n,k)\in [a,a+1] \big) \le C \big( \log \frac nk \big) e^{C/\ep}\P(T(n,k)\le a+1)^{1-\ep}\sqrt{\frac kn}+\left(\frac kn\right)^{2} .
    \end{equation}
\end{prop}

    

The proofs in this section are similar to those in the previous section, with the key difference being the approach used to control how the perturbation affects $T(n,k)$, which requires new methods. Specifically, a major challenge arises when perturbing atomic distributions. For continuous distributions, we could argue that the probability of significant perturbation for each edge weight was high. This allowed us to ensure that all paths must pass through a substantial fraction of edges that were significantly perturbed. However, for atomic distributions, we can no longer guarantee that each edge has a high probability of being affected by the perturbation. To address this, rather than proving a uniform bound on the impact of the perturbation across all paths, we focus on how the reduction in edge weights affects the geodesic. To ensure that the perturbation sufficiently impacts $T(n,k)$, we need to control the number of $b$-edges on the geodesic, as these are the only edges whose values can be reduced during the perturbation. This is the objective of Proposition \ref{cor:lotof b edges}.

\begin{prop}\label{cor:lotof b edges}
There exist positive constants $\alpha, c$ and $C$ such that for any $n\ge k\ge 1$ and  $n\ge 1$, we have
\begin{equation*}
    \P \big(\text{$\exists \gamma$ a geodesic for $T(n,k)$ such that } |\{e\in \gamma : t_e=b\}|\ge \alpha n \big) \le Ce^{-cn}.
\end{equation*}
\end{prop}

This statement essentially follows from a result of van den Berg and Kesten \cite{van1993inequalities} and of Marchand \cite{marchand2002strict}, that the time constant increases as the distribution becomes more variable. See also \cite[Theorem~2.12]{50years} and the discussion after the theorem for a nice explanation about this result. However, to get Proposition~\ref{cor:lotof b edges} as stated, a slightly technical proof is required in order to show that restricted, left-right crossing has the same time constant as non-restricted, point to point passage time. We postpone its proof to the end of the section.

    


To define the perturbed environment for this context, we build the randomness by using a family of standard Gaussian random variables. More precisely, consider a family of standard Gaussian random variables $(N_e)_e$ and set \[t_e(N_e)\coloneqq a\mathds{1}_{N_e\le s}+b\mathds{1}_{N_e> s}\]
where $s$ is chosen such that $\P(N_e\le s)=G(\{a\})$. 
For $r\in [-2,2]$ we define the perturbation function $\tau _r:\mathbb E(n,2k)\to [0,\infty )$ by
\begin{equation}
    \tau _r (e):= \frac{r}{\sqrt{kn}}, 
\end{equation}

We define the modified environment by \[t_{e,r}:=t_e(N_e+\tau_r(e)).\] We denote by $T_r(n,k)$  the passage time of minimal left-right crossing in the rectangle $[0,n]\times [0,2k]$ and vertical displacement at most $k$ in the modified environment $(t_{e,r})$. We will continuously vary $r$ and analyse the effect it has on the passage time $T_r(n,k)$.

We begin by giving a lower bound on the impact of the perturbation  when $r<0$ on $T(n,k)$. When $r$ is negative, the modification increases the number of $a$ edges. The purpose of this lemma is to ensure that there is a big amount of edges on the geodesic whose value will decrease after the modification.

\begin{lem}\label{lem:1bis} There exists $c>0$ such that
for any $n\ge k\ge 1 $ and $r\in [0,1]$, we have
\begin{equation}
\mathbb P \Big( T(n,k)-T_{-r}(n,k)\le cr\sqrt{\frac n k}  \Big) \le e^{-cr \sqrt{\frac nk}}.   \end{equation}
\end{lem}

We will also need the following large deviation estimate for the binomial distribution.
\begin{claim}{\cite[Theorem A.1.13]{alon2016probabilistic}} \label{claim:bin} For any  $p\in(0,1/2)$ and integer $n\ge 1$ we have
\begin{equation}
    \P \big( \Bin(n,p)\le \frac{np}2 \big) \le e^{-\frac{np}8}.
\end{equation}
\end{claim}

\begin{proof}[Proof of Lemma \ref{lem:1bis}]
    Let $r\in [0,1]$.
    For Gaussian random variables, we can choose $g_\tau(x)=x+\tau$ in Lemma \ref{lem:MW} (we refer to Claim 2.13 in \cite{dembin2024coalescence}).
    Condition on $(t_e)_e$, note that 
    \begin{equation}
    \begin{split}
        \P \big( t_{e,-r}=a \mid t_e=b \big) &=\P \big( N_e+\tau_r(e)\le s \mid N_e>s \big) \\
        &\ge \frac 1{G(\{b\})}\P \big(  s< N_e\le s+\tau_r(e) \big) \ge c  \frac{r}{\sqrt{kn}}.
        \end{split}
    \end{equation}
    Now note that the random variables $(t_{e,-r})_e$ are still independent when conditioning on $(t_e)_e$.
    Denote by $\gamma$ a geodesic for $T(n,k)$. By Proposition \ref{cor:lotof b edges}, with high probability $\gamma$ contains at least $\alpha n$ edges with value $b$.
    Let $E:=\{e\in \gamma : t_e =b\}$ be this random set of edges.  By Claim~\ref{claim:bin}, we have
    \begin{equation}\label{eq:bin}
        \P\Big(\sum_{e\in E}\mathds{1}_{t_{e,-r}=a}\le \frac c 2 \frac{r}{\sqrt{kn}}|E| \ \big| \ (t_e)_e\Big)\le e^{-\frac c 8\frac{r}{\sqrt{kn}}|E|}.
    \end{equation}
    Note that on the event
    \[\{|E|\ge \alpha n\}\cap\left\{\sum_{e\in E}\mathds{1}_{t_{e,-r}=a}\ge \frac c 2\frac{r}{\sqrt{kn}}|E|\right\},\]
    we have
    \begin{equation*}
        T_{-r}(n,k)\le \sum_{e\in \gamma}t_{e,-r}\le \sum_{e\in \gamma}t_{e}-(b-a)\frac c 2\frac{r}{\sqrt{kn}}|E|\le T(n,k)-\alpha\frac c 2 (b-a)r\sqrt{\frac n k}.
    \end{equation*}
    The result follows thanks to inequality \eqref{eq:bin} together with Proposition \ref{cor:lotof b edges}.
\end{proof}
The following corollary can be proved similarly as Corollary \ref{cor:2}.
\begin{cor}\label{cor:2bis}There exists $c>0$ such that for any $s\in[-1,1]$
for any $ r\in [0,1]$ we have that 
\begin{equation}
\mathbb P \Big( T_s(n,k)-T_{s-r}(n,k)\le cr\sqrt{\frac n k}  \Big) \le e^{-cr \sqrt{\frac nk}}.   \end{equation}
    
\end{cor}
The proof of Proposition \ref{prop:MW2} will follow from the next lemma with a similar proof as in the proof of Proposition \ref{prop:MW}.

\begin{lem}\label{cor:1bis}
    There exists $C>0$ depending only on $G$ such that for any $a\ge 0$ and $\ep>0$ we have 
    \begin{equation}
        \int _{-1}^1 \mathbb P \big(  T_r(n,k) \in [a,a+1] \big) dr \le Ce^{C/\ep}\left(\log \frac nk\right) \P(T(n,k)\le a+1)^{1-\ep}\sqrt{\frac kn}+\left(\frac kn\right)^{2} .
    \end{equation}
\end{lem}

\begin{proof} Let $c$ be the constant from Corollary \ref{cor:2bis} and let $C$ be a large enough constant depending on $c$.
Let $r_0:=  C\log \frac nk\sqrt {\frac k n} $  and let $R:=r_0\mathbb Z \cap [-1,1]$. Define the event
\begin{equation}
    \Omega :=\bigcap _{r\in  R}  \Big\{ T_r(n,k)-T_{r-r_0}(n,k)\ge cC\log \frac nk\Big\}.  
\end{equation}
Using Corollary~\ref{cor:2bis} and a union bound we obtain for $C$ large enough
\begin{equation}\label{eq:Omegabis}
\mathbb P (\Omega ^c )\le  C \sum _{r\in R} e^{-cC\log \frac nk}   \le \sqrt{\frac nk }e^{-cC\log \frac nk}  \le \left(\frac kn\right)^{2} .
\end{equation}
Next, we write 
\begin{equation}
        \int _{-1}^1 \mathbb P \big(  T_r(n,k) \in [a,a+1] \big) dr \le \mathbb P (\Omega ^c)+  \int _{-1}^1 \mathbb P \big( \Omega , \ T_r(n,k) \in [a,a+1] \big) dr.
    \end{equation}
Let us bound the second term on the right hand side of \eqref{eq:1}.  Using that $T_r(n,k)$ is increasing in $r$ we obtain that on $\Omega$ there is at most one element $r\in R$ for which $T_r(n,k)\in [a,a+1]$ and that \[\big| \big\{r\in [-1,1] : T_r(n,k)\in [a,a+1] \big\}\big| \le C \Big( \log \frac nk \Big)  \,\sqrt{\frac k n},\] where in here $|\cdot |$ denotes the Lebesgue measure of the set. Moreover, it is easy to check that since $T_r(n,k)$ is increasing in $r$, we have 
 \begin{equation*}
     \big| \big\{r\in [-1,1] : T_r(n,k)\in [a,a+1] \big\}\big|=\big| \big\{r\in [-1,1] : T_r(n,k)\in [a,a+1] \big\}\big|\mathds{1}_{T_{-1}(n,k)\le a+1}.
 \end{equation*}
 Hence, using Fubini's theorem
    \begin{equation}\label{eq:1bis2}
    \begin{split}
        \int _{-1}^1 \mathbb P \big( \Omega , \ T_r(n,k) \in [a,a+1] \big) dr&= \mathbb E \Big[ \mathds 1 _{\Omega }\big| \big\{r\in [0,1] : T_r(n,k)\in [a,a+1] \big\}\big| \Big] \\&\le C\P \big( T_{-1}(n,k)\le a+1 \big)\Big( \log \frac nk \Big)\sqrt{\frac kn}.
        \end{split}
    \end{equation}
   We obtain similarly as in \eqref{eq:MWT-1} that 
   \begin{equation*}
    \P(T_{-1}(n,k)\le a+1)\le e^{C/\ep}\P(T(n,k)\le a+1)^{1-\ep}.
\end{equation*}
Substituting this estimate into \eqref{eq:1bis2} finishes the proof of the lemma.
\end{proof}

\begin{proof}[Proof of Proposition \ref{cor:lotof b edges}]
    For $e\in E(\mathbb Z ^2)$, define the new weights $t_e^+:=t_e+\mathds 1\{t_e=b\}$. Let $\mu, \mu ^+ :\mathbb R ^2\to \mathbb [0,\infty ]$  be the limiting norms for the old and new environments respectively. By \cite[Theorem~2.12]{50years} we have that $\mu ^+(e_1)>\mu (e_1)$. We will prove the estimate of the proposition with $\alpha := (\mu ^+(e_1)-\mu (e_1))/4$. We have that 
    \begin{equation}
        \mathbb P \big( \text{there exists a geodesic $\gamma $ with }   |\{e\in \gamma :t_e=b\}|\le \alpha n  \big) \le \mathbb P \big( T^+(n,k) \le T(n,k)+\alpha n\big),
    \end{equation}
    where $T^+(n,k)$ is defined as $T(n,k)$ but with the modified environment $(t_e^+)$. Letting $T^+_{x,y}$ be the (unrestricted) passage time in the environment $(t_e^+)$ between $x$ and $y$ we have that $T^+(n,k)\ge \min _{x,y} T^+_{x,y}$ where the minimum is taken over all pairs of points $x\in \{0\}\times [0,2k]$ and $y\in \{n\}\times [0,2k]$. 
    
    Next, we obtain an upper bound on $T(n,k)$. Let $k_0$ sufficiently large such that for all $r\ge k_0$ we have $\mathbb E [T_{0,(r,0)}] \le r(\mu (e_1) +\alpha )$. Let $0\le s_0\le \cdots \le s_m=n$ be a sequence of integers such that for all $i\le m$ we have $k_0 \le s_i-s_{i-1} \le 2k_0$. 
    
    Suppose first that $k\ge 4k_0b/a$. We claim that in this case, we have that $T(n,k)\le \Tilde{T}$ where $\tilde{T}:=\sum _{i=1}^m T_{(s_{i-1},k),(s_{i},k)}$. Indeed, a geodesic from $(s_{i-1},k)$ to $(s_i,k)$ is almost surely of length at most $2k_0b/a$. Thus, concatenating all these geodesics we obtain a path from $(0,k)$ to $(n,k)$ which is almost surely contained in $[-2k_0b/a,n+2k_0b/a]\times [k-2k_0b/a,k+2k_0b/a]$. This path might exit $[0,n]\times [0,2k]$ only near the endpoints but clearly there is a subpath of it which is a proper left-right crossing of $[0,n]\times [0,2k]$ with vertical fluctuations bounded by $4k_0b/a \le k$. Hence, $T(n,k)\le \Tilde{T}$. Substituting the lower bound on $T^+(n,k)$ and the upper bound on $T(n,k)$ we obtain
    \begin{equation}\label{eq:778}
       \mathbb P \big( T^+(n,k) \le T(n,k)+\alpha n\big) \le \sum _{x,y} \mathbb P \big( T^+_{x,y} \le \tilde{T} +\alpha n \big) ,
    \end{equation}
where the sum is over pairs $x\in \{0\}\times [0,2k]$ and $y\in \{n\}\times [0,2k]$. To bound the last probability note that by the definition of $k_0$ we have
\begin{equation}
    \mathbb E [\tilde{T}] \le n(\mu (e_1)+\alpha ) \le n(\mu ^+(e_1)-3\alpha ) \le \mu ^+(x,y)-3\alpha n,
\end{equation}
Where the last inequality holds for all $x\in \{0\}\times [0,2k]$ and $y\in \{n\}\times [0,2k]$ since the limiting norm is symmetric and convex. Thus, for all such $x,y$ we have
    \begin{equation}\label{eq:777}
      \mathbb P \big( T^+_{x,y} \le \tilde{T} +\alpha n \big) \le  \mathbb P \big( |T^+_{x,y}-\mu ^+(x,y)|\ge \alpha n\big) +\mathbb P \big( |\tilde{T}-\mathbb E [\tilde{T}] |\ge \alpha n ).
    \end{equation}
    The first probability on the right hand side of \eqref{eq:777} is exponentially small by Theorem~\ref{thm:Talagrand}. The second probability on the right hand side of \eqref{eq:777} is exponentially small by Azuma's inequality as $\tilde{T}$ depends only on the weights in $[-2k_0b/a,n+2k_0b/a]\times [k-2k_0b/a,k+2k_0b/a]$ as explained above (so $\tilde{T}$ is a Lipschitz function of $O(n)$ i.i.d.\ variables). Substituting these bounds in \eqref{eq:777} and then in \eqref{eq:778} finishes the proof of the corollary when $k\ge 4k_0b/a$.

    Next, suppose that $k\le 4k_0b/a$. In this case, with exponentially high probability, there are linearly many columns of horizontal edges in $[0,n]\times [0,2k]$ with all edges in the column having weight $b$. When this event happens, any left-right crossing will contain linearly many edges of weight $b$.
\end{proof}

\section{Lower bounding the variance}\label{sec 3}
The aim of this section is to prove Theorem \ref{thm:lowerboundvar}.
The idea of the proof is that $\tau(n,k)$ can be written as a minimum of an order $n/k$ of random variables distributed as $T(n,k)$. We can prove that the $\alpha$-quantile $q_\alpha(\tau(n,k))$ of $\tau(n,k)$ is a tail value for $T(n,k)$. Using small ball probability upper-bound in the tail for $T(n,k)$, we can obtain a small ball probability bound for $\tau(n,k)$ and deduce from it a lower bound on the variance.

\begin{lem}\label{lem:cormw}
    Let $n\ge k \ge 1$ and $\alpha,\ep\in(0,1)$, we have
     \begin{equation}\label{eq:lem311}
        \mathbb P \big( T(n,k)<q_\alpha(\tau(n,k))\big) \le \frac{4k}n|\log(1-\alpha)|
    \end{equation}
    and for any $x\in\mathbb R$
    \begin{equation}\label{eq:lem312}
    \P(\tau(n,k)\in[x,x+1])\le 3\frac n k \P(T(n,k)\in[x,x+1]) .
\end{equation}
\end{lem}
\begin{proof}
    For any rectangle $R=[0,n]\times [s,t]$ with $s\le t$, we define $T_k(R)$ to be the smallest passage time between the left and the right side of $R$ with vertical displacement at most $k$.
It is easy to check that there exists $m\le 3n/k$ such that we can find $m$ translates $R_1,\dots,R_m$ of the rectangle $[0,n]\times [0,2k]$ such that
\[\tau(n,k)=\min_{i\in\{1,\dots,m\}}T_k(R_i).\]
To see that it is enough to consider the $n/2k$ disjoint rectangles of width $2k$ and their translates by $k$ strictly included in $[0,n]^2$. More precisely, set $m_0=\lfloor n/2k\rfloor$ and consider $[0,n]\times [2ki,2k(i+1)]$,  $0\le i\le m_0-1$, $[0,n]\times [n-2k,n]$, $[0,n]\times [2ki+k,2k(i+1)+k]$,  $0\le i\le m_0-2$ and $[0,n]\times [n-3k,n-k]$. For this family, any path in $\cP_k(n)$ has to lie strictly in at least one of the $R_i$ yielding
\[\tau(n,k)\ge \min_{i\in\{1,\dots,m\}}T_k(R_i).\]
Conversely, since the $R_i$ are strictly contained in $[0,n]^2$, it gives the converse inequality.  

We have
\[\P(\tau(n,k)\ge  q_\alpha(\tau(n,k)))\ge 1-\alpha.\]
Since, we can find $\lfloor m_0/2\rfloor$ disjoint translates of $R_1$, we have
\begin{equation*}
    \P\big( \tau(n,k)\ge  q_\alpha(\tau(n,k)) \big) \le \P \big( T_k(R_1)\ge q_\alpha(\tau(n,k)) \big)^{n/4k} 
\end{equation*}
Hence,
\begin{equation}\label{eq:quantile estimate}
    \P(T_k(R_1)\ge q_\alpha(\tau(n,k)))\ge \exp\left(\frac {4k} n \log(1-\alpha)\right)\ge 1- \frac {4k} n |\log(1-\alpha)|.
\end{equation}
 This completes the proof of \eqref{eq:lem311}. 

Since the geodesics between left-right with constrained vertical fluctuations have to lie strictly in one of the rectangle $R_1,\dots, R_m$, we have that
\begin{equation*}
    \{\tau(n,k)\in[x,x+1]\}\subset \bigcup_{i=1,\dots,m}\{T_k(R_i)\in[x,x+1]\}
\end{equation*}
yielding the following bound
\begin{equation*}
    \P(\tau(n,k)\in[x,x+1])\le 3\frac n k \P(T_k(R_1)\in[x,x+1]) .
\end{equation*}
\end{proof}

\begin{proof}[Proof of Theorem \ref{thm:lowerboundvar}]
We assume here that $G$ satisfies assumption \eqref{eq:assumption abs1}.
Let $x\le q_\alpha(\tau(n,k))-1$.
By Proposition \ref{prop:MW}, we have
 \begin{equation}
        \mathbb P \big( T(n,k)\in [x,x+1] \big) \le Ce^{C/\ep}\sqrt{\frac k n} \mathbb P \big( T(n,k)\le x+1 \big)^{1-\ep}+e^{-c\sqrt n}.
    \end{equation}
It yields combining the previous inequality together with Lemma \ref{lem:cormw} \eqref{eq:lem311}
   \begin{equation}
        \mathbb P \big( T(n,k)\in [x,x+1] \big) \le Ce^{C/\ep}\left(\frac k n\right)^{\frac 32- \ep}|\log(1-\alpha)|+e^{-c\sqrt n}.
    \end{equation}
It follows using Lemma \ref{lem:cormw} \eqref{eq:lem312}, for $n$ large enough depending in $\alpha$
\begin{equation*}
    \P(\tau(n,k)\in[x ,x+1))\le \left(\frac kn\right)^{\frac 12-\ep} Ce^{C/\ep}|\log(1-\alpha)|.
\end{equation*}
For short write $q_\alpha$ for $q_\alpha(\tau(n,k))$.
By applying the previous inequality several times for $x\in [q_\alpha-ce^{-C/\ep}\left(\frac nk\right)^{\frac 12-\ep} ,q_\alpha-1]\cap \Z$, we have
\begin{equation*}
    \P(\tau(n,k)\in[q_\alpha-ce^{-C/\ep}\left(\frac nk\right)^{\frac 12-\ep} ,q_\alpha])\le |\log(1-\alpha)|.
\end{equation*}
Hence, we have by choosing $\alpha=1/2$
\begin{equation*}
    \var(\tau(n,k))\ge ce^{-C/\ep}\left(\frac nk\right)^{1-2\ep}.
\end{equation*}
Finally, by optimizing in $\ep$ we get 
\begin{equation*}
    \var(\tau(n,k))\ge ce^{-C\sqrt{\log \frac nk }}\left(\frac nk\right).
\end{equation*}
Let us now assume \eqref{eq:assumption atomic}. By Proposition \ref{prop:MW2}, we have
\[ \mathbb P \big( T(n,k)\in [x,x+1] \big) \le C\log \frac nk e^{C/\ep}\P(T(n,k)\le x+1)^{1-\ep}\sqrt{\frac kn}+\left(\frac kn\right)^{2}.\]
By similar computations as above, we get
\begin{equation*}
    \var(\tau(n,k))\ge ce^{-C/\ep}\left(\log \frac nk\right)^{-1}\left(\frac nk\right)^{1-2\ep}.
\end{equation*}
Optimizing in $\ep$ yields
\begin{equation}\label{eq:1 var}
    \var(\tau(n,k))\ge ce^{-C\sqrt{\log \frac nk }}\left(\frac nk\right),
\end{equation}
finishing the proof of the first part.

For the second part of the theorem we further assume that assumptions \eqref{ass:uniform curbature} and \eqref{eq:assumption exp1} hold.
Let $\ep>0$. Denote for short ($\tau$, $T$) and $(\tau',T')$ two independent copies of $(\tau(n,n^{3/4+\ep}),T_n)$.
We have
\begin{equation}
    \begin{split}
2\var(\tau(n,n^{3/4+\ep}))&=\E[(\tau-\tau')^2]\le\var(T_n)+ 2\E[(\tau^2+\tau'^2)(\mathbf{1}_{T\ne \tau}+ \mathbf{1}_{T'\ne \tau'})]\\
    &\le \var(T_n)+8\sqrt{\E[\tau(n,n^{3/4+\ep})^4]}\sqrt{\P(T_n\ne\tau(n,n^{3/4+\ep})) }
    \end{split}
\end{equation}
where we used Cauchy-Schwarz in the last inequality.
We conclude using Proposition \ref{prop:6.2} and \eqref{eq:1 var}.
\end{proof}

\section{Noise sensitivity for being above a given quantile}\label{sec 4}
In this section, we prove Theorem \ref{thm:noisesensitivity}.
Consider in this section distributions $G$ that satisfy \eqref{eq:assumption atomic}.
Let $\alpha\in(0,1)$. Instead of directly proving that the sequence $(\mathds{1}_{\tau(n,k)\le q_\alpha(\tau(n,k))})$ is noise sensitive, we will first prove that another sequence with more symmetry is noise sensitive by using BKS theorem and then prove that the values of these two sequences agree with high probability.
Define $\widetilde \tau(n,k)$ to be the value corresponding to $\tau(n,k)$ when the top and the bottom of the square $[0,n]^2$ are glued (we identify the vertices $(i,0)$ and $(i,n)$ for $i\in\{0,\dots,n\}$). Recall that $\mathcal P_k(n)$ is the set of paths from left to right in $[0,n]^2$ with vertical displacement at most $k$. Analogously, 
$\widetilde \cP_k(n)$ denotes the set of paths from the left to the right of the cylinder $[0,n]\times \Z/n\Z$ with vertical displacement at most $k$. 
To assign a weight to a path $p\in  \widetilde \cP_k(n)$, we assign the weight $t_e$ when $e$ is an edge in $[0,n]\times [0,n-1]$ and we assign the weight  $t_{\{(n-1,j),(n,j)\}}$ for the edge $\{(0,j),(n-1,j)\}$, $j\in\{0,\dots,n\}$. 

The random variable $\widetilde\tau(n,k)$ has more symmetry. Specifically, the influence of two edges that are vertical translations of each other on the function $\mathds{1}_{\widetilde \tau(n,k) \le q_\alpha(\tau(n,k))}$ is the same. Theorem \ref{thm:noisesensitivity} will follow easily from the following proposition and lemmas.
We first prove that the auxiliary sequence is noise sensitive using BKS theorem (see Proposition \ref{prop:noise sensitive}) then we prove that with high probability $\tau(n,k)=\widetilde \tau(n,k)$
(see Lemma \ref{lem:seq agree}). Finally, we prove that if a sequence is equal, with high probability, to a sequence that is noise sensitive, it must also be noise sensitive (see Lemma \ref{lem:noise sensitive}).
\begin{prop}\label{prop:noise sensitive}There exists $C>0$ such that the following holds. For any $\alpha\in(0,1)$. Let $n\ge 1$ and $(k_n)_{n\ge1}$ be a sequence such that $k_n\le n^{1/2}e^{-C\sqrt {\log n}}$. The sequence $(\mathds{1}_{\widetilde \tau(n,k_n)\le q_\alpha(\tau(n,k))})_n$ is noise sensitive.

For any $\epsilon >0$ there exists $s>0$ such that assuming \eqref{ass:sides} then for $k_n\le n^{1-\epsilon }$ the sequence $(\mathds{1}_{\tau(n,k_n)\le q_\alpha(\tau(n,k_n))})_{n\ge 1}$ is noise sensitive.
    \end{prop}
    Then we prove that the sequences agree with high probability and that we can deduce the noise sensitivity of the original sequence. 
    \begin{lem}\label{lem:seq agree}
Let $n\ge k\ge1$.
    \begin{equation}
        \P \big( \tau(n,k)\ne\widetilde \tau(n,k) \big) \le C\frac k n .
    \end{equation}
\end{lem}
\begin{lem}\label{lem:noise sensitive}Let $(f_n)_n$ be a noise sensitive sequence and let $(g_n)_n$ be such that \[\lim_{n\rightarrow\infty}\P(f_n\ne g_n)=0.\]
Then, the sequence $(g_n)_n$ is noise sensitive.
    
\end{lem}
The third part of Theorem \ref{thm:noisesensitivity} follows easily by combining Propositions \ref{prop:noise sensitive},  \ref{prop:6.2} and Lemma \ref{lem:noise sensitive}.

We start by proving Proposition \ref{prop:noise sensitive}.
To prove this proposition we will need the two following lemmas.

For $e\in\mathbb E^2$, we define $\mathfrak C(e)$ to be the set of vertical translates of $e$ contained in the cube
$\mathfrak C(e)\coloneqq \{e+k\mathrm e_2: e+k\mathrm e_2\in[0,n]^2\}$. As explained in the introduction, the main obstacle to prove noise sensitivity for any $k\le n e^{-C\sqrt {\log n}}$ is if the geodesic has a too large intersection with $\mathfrak C(e)$. This is something that is not expected to occur though we cannot exclude it without assumptions.

 The following lemma states that the geodesic does not spend too much time in a given column. It is proved using Proposition \ref{cor:number of sides}.
  For short, we write $q_\alpha$ for $q_\alpha(\tau(n,k))$.
\begin{lem}\label{lem:intersection column} For any $\delta>0$, there exist $s\ge 1$ and $C_\delta,c_\delta>0$ such that under assumption \eqref{ass:sides},
    for any edge $e$ in $[0,n]^2$, we have for any $1\le k \le n/4$
 \begin{equation}
    \mathbb E[|\pi\cap \mathfrak C(e)|\mathds {1}_{\widetilde\tau(n,k)\in[q_\alpha-b+a,q_\alpha]}]\le Cn^\delta\P(\widetilde\tau(n,k)\in[q_\alpha-b+a,q_\alpha]) +C_\delta e^{-n^{c_\delta }}
 \end{equation}
 where $\pi$ is the intersection of all paths achieving the infimum in $\widetilde\tau(n,k)$.
\end{lem}

\begin{proof}
For $i\le n$, let $R_i$ be the rectangle $[0,n]\times [0,k]$ that is shifted cyclically up by $i$ units in the cylinder $[0,n]^2$. Namely, this is the rectangle whose bottom row is $i$ and whose top row is $k+i(\!\!\!\mod n)$. Let $\pi _i$ be the intersection of all minimal left-right crossing of the rectangle $R_i$. By the definition of the passage time $\widetilde\tau(n,k)$, the intersection $\pi $ is equal to the intersection $\pi_i$ for some $i\le n$ and therefore by Proposition~\ref{cor:number of sides} and a union bound over $i\le n$ we have that 
    \begin{equation}
        \mathbb P \big( |\pi\cap \mathfrak C(e)| \ge n^{\delta } \big)<C_\delta  e^ {-n^{c_\delta }},
    \end{equation}
    as long as \eqref{ass:sides} holds for a sufficiently large $s$ depending on $\delta $. The statement of the lemma follows immediately from the last estimate.
\end{proof}

To prove the Proposition \ref{prop:noise sensitive}, we will need the following optimization result.
\begin{claim}\label{claim:optimization}For any $B\ge A>0$ and integer $m\ge 1$ we have
    \[\sup\left \{\sum_{i=1}^m x_i^2: 0\le x_i\le A, \sum_{i=1}^m x_i\le B\right\}\le \left(\frac {B} A+1\right) A^2\le 2AB .\]
\end{claim}
The value \( x_i \) represents the influence of an edge; it roughly corresponds to the probability that the edge lies on the geodesic, and the value of \( \widetilde{\tau}(n,k) \) is close to \( q_\alpha \). We will use the vertical symmetry together with the bound on the expected intersection with a column to get a bound on $A$. We will further use a global bound by using the bound on the size of the geodesic to derive a bound on $B$.

\begin{proof}[Proof of Proposition \ref{prop:noise sensitive}]
   We aim to use BKS theorem, so we need to bound the influence of the edges.
Since the fact that $e$ is pivotal does not depend on the value of $t_e$, we have
\begin{equation}\label{eq:6}
\begin{split}
    \Inf_e(f_n)&=\frac 1 {G(\{a\})}\P(\text{$e$ is pivotal and $t_e=a$})\\&\le \frac 1 {G(\{a\})}\P(\text{$e\in\pi$ and $\widetilde \tau(n,k)\in[q_\alpha-b+a,q_\alpha]$})
    \end{split}
\end{equation}
where $\pi$ is the intersection of all geodesics achieving the infimum in $\widetilde \tau(n,k)$.

Using the vertical symmetry, we obtain
\begin{equation}
\begin{split}
    \P(\text{$e\in\pi$ and $\widetilde \tau(n,k)\in[q_\alpha-b+a,q_\alpha]$})&=  \frac{1}{n}\sum_{f\in\mathfrak C(e)}\P(\text{$f\in\pi$ and $\widetilde \tau(n,k)\in[q_\alpha-b+a,q_\alpha]$})\\
    &= \frac 1n  \mathbb E[|\pi\cap \mathfrak C(e)|\mathds {1}_{\widetilde\tau(n,k)\in[q_\alpha-b+a,q_\alpha]}].
\end{split}
\end{equation}
Next, assume that for some $K_n$ we have
 \begin{equation}
    \mathbb E[|\pi\cap \mathfrak C(e)|\mathds {1}_{\widetilde\tau(n,k)\in[q_\alpha-b+a,q_\alpha]}]\le CK_n\P(\widetilde\tau(n,k)\in[q_\alpha-b+a,q_\alpha]).
 \end{equation}
 Trivially, we can choose $K_n=k$ but in some situations we have a better bound (e.g., under the assumption \eqref{ass:sides} or in directed models). It yields that
 \begin{equation}
    \P(\text{$e\in\pi$ and $\widetilde \tau(n,k)\in[q_\alpha-b+a,q_\alpha]$})\le C\frac {K_n}n  \P(\widetilde\tau(n,k)\in[q_\alpha-b+a,q_\alpha])\coloneqq A.
\end{equation}
It is easy to check that
\begin{equation}
\begin{split}
    \sum_{e} \P(\text{$e\in\pi$ and $\widetilde \tau(n,k)\in[q_\alpha-b+a,q_\alpha]$})&=\mathbb E[|\pi|\mathds {1}_{\widetilde \tau(n,k)\in[q_\alpha-b+a,q_\alpha]}]\\&\le \frac ba n
\P(\widetilde \tau(n,k)\in[q_\alpha-b+a,q_\alpha])\coloneqq B. 
\end{split}
\end{equation}
Using \eqref{eq:6} and Claim \ref{claim:optimization} we obtain
\begin{equation}
\begin{split}
     \sum_e  \Inf_e(f_n)^2&\le \frac 1 {G(\{a\})^2 }\sum_{e} \P(\text{$e\in\pi$ and $\widetilde \tau(n,k)\in[q_\alpha-b+a,q_\alpha]$})^2\\
     &\le \frac 1 {G(\{a\})^2 } 2AB\le CK_n\P(\widetilde \tau(n,k)\in[q_\alpha-b+a,q_\alpha])^2.
\end{split}
\end{equation}
 Besides, using Lemma \ref{lem:cormw}\footnote{The statement of Lemma \ref{lem:cormw} holds for \(\tau\), but with a straightforward adaptation of the proof, one can check that the same statement also holds for \(\widetilde{\tau}\) in place of \(\tau\).} and Proposition \ref{prop:MW2} yields for $n$ large enough depending on $\alpha$
 \begin{equation}
     \P(\widetilde \tau(n,k)\in[q_\alpha-b+a,q_\alpha])\le \frac {3n}k\P(T(n,k)\in[q_\alpha-b+a,q_\alpha])\le  Ce^{C/\ep}\log \frac nk \left(\frac k n\right)^{\frac 12- \ep}|\log(1-\alpha)|.
 \end{equation}
 Combining the two previous inequalities, we obtain
\begin{equation}
     \sum_e  \Inf_e(f_n)^2\le CK_ne^{C/\ep}\left(\log \frac nk\right)^2 \left(\frac k n\right)^{1-2 \ep}|\log(1-\alpha)|^2
\end{equation}
We choose $\ep= (\log \frac n k)^{-\frac 12 }$.
It yields that 
\begin{equation*}
    \sum_e  \Inf_e(f_n)^2\le C e^{C\sqrt{\log \frac nk}} |\log(1-\alpha)|^2 \cdot\left\{\begin{array}{ll} \frac k n &\mbox{if $K_n\le C$}\\
    \frac {k^2} n &\mbox{if $K_n\le k$}\\
   \frac k {n^{1-\delta}} &\mbox{if $K_n\le n^\delta$}\end{array}\right..
\end{equation*}
The first case is true in particular for an oriented version of the model. We expect it to be true in general.
The second bound corresponds to the crude upper bound using that geodesics for $\tau(n,k)$ have vertical fluctuations at least $k$.
The last bound corresponds to what we get when we assume \eqref{ass:sides}. 
Hence the sum goes to $0$ when $k\le n^{1/2}e^{-C\sqrt{\log n}}$ since we always have $K_n\le k$.

By Lemma~\ref{lem:intersection column}, if we further assume \eqref{ass:sides} for $s$ large enough depending on $\delta$,
 \begin{equation}
    \mathbb E[|\pi\cap \mathfrak C(e)|\mathds {1}_{\widetilde\tau(n,k)\in[q_\alpha-b+a,q_\alpha]}]\le Cn^\delta\P(\widetilde\tau(n,k)\in[q_\alpha-b+a,q_\alpha]) +C_\delta e^{-n^{c_\delta }}
 \end{equation}
 If $Cn^\delta\P(\widetilde\tau(n,k)\in[q_\alpha-b+a,q_\alpha])\le C_\delta e^{-n^{c_\delta }} $, then we obtain the bound
\[\sum_e  \Inf_e(f_n)^2\le n^2\P(\widetilde\tau(n,k)\in[q_\alpha-b+a,q_\alpha])^2=o(  e^{-\frac 12n^{c_\delta }}). \]
Otherwise, we get $K_n\le C n^\delta$. Hence, in both cases, we have
\[ \sum_e  \Inf_e(f_n)^2\le C e^{C\sqrt{\log \frac nk}} |\log(1-\alpha)|^2\frac k {n^{1-\delta}}.\]
It follows that the sum goes to $0$ for $k\le n^{1-2\delta}$. The result follows.
\end{proof}

\begin{proof}[Proof of Lemma \ref{lem:noise sensitive}]Let $\ep>0$. We have
\begin{equation}
    \begin{split}
        \EE[g_n(t)g_n(t^\ep)]-\EE[g_n(t)]^2&=\EE[g_n(t)g_n(t^\ep)]-\EE[f_n(t)f_n(t^\ep)]\\&\quad+ \EE[f_n(t)]^2-\EE[g_n(t)]^2\\&\quad+\EE[f_n(t)f_n(t^\ep)]-\EE[f_n(t)]^2.
    \end{split}
\end{equation}
We have
 \begin{equation}
       | \EE[g_n(t)g_n(t^\ep)-f_n(t)f_n(t^\ep)]|\le \P(g_n(t)\ne f_n(t))+\P(g_n(t^\ep)\ne f_n(t^\ep)).
 \end{equation}
 Similarly
 \begin{equation}
     |\EE[f_n(t)]^2-\EE[g_n(t)]^2|=|\EE[f_n(t)-g_n(t)]|\EE[f_n(t)+g_n(t)] \le 2\P(g_n(t)\ne f_n(t)).
 \end{equation}
 Combining all the previous inequalities, we get that
 \begin{equation}
   \lim_{n\rightarrow\infty}  \EE[g_n(t)g_n(t^\ep)]-\EE[g_n(t)]^2=0
 \end{equation}
 concluding the proof.
\end{proof}

\begin{proof}[Proof of Lemma \ref{lem:seq agree}]
Let $\widetilde\gamma$ be a geodesic for $\widetilde \tau(n,k)$. If $\widetilde\gamma\in\cP_k(n)$ (that is, $\widetilde\gamma$ does not have edges of the form $\{(i,0),(i,n-1)\}$), then we have
\[\widetilde\tau(n,k)\ge \tau(n,k).\]
Conversely, let $\gamma $ be a geodesic for $\tau(n,k)$. If $\gamma$ does not take edges of the top (that is of the form $\{(i,n), (i+1,n)\}$), then we have
that $\gamma \in \widetilde \cP_k(n)$ and
\[\widetilde\tau(n,k)\le \tau(n,k).\]
It yields that
\begin{equation}\label{eq:inclusion}
    \{\gamma\in \widetilde \cP_k(n)\}\cap\{\widetilde \gamma\in\cP _k(n)\}\subset \{\tau(n,k)=\widetilde\tau(n,k)\}.
\end{equation}

Denote by $R$ the rectangle that is a translate of $[0,n]\times[0,2k]$ such that its middle corresponds to the top. In particular, if $\gamma\notin  \widetilde \cP_k(n)$ (respectively $\widetilde \gamma\notin\cP_k(n)$) then $\gamma\subset R$ (respectively $\widetilde\gamma\subset R$ where $R$ is seen in $[0,n]\times \Z/n\Z$).
 As in the proof of Lemma \ref{lem:cormw}, let $m_0=\lfloor n/2k\rfloor$ we can find a family $(R_i)_{1\le i \le m_0}$ of disjoint translates of $R$ such that $R_1=R$ (in the case of $\widetilde \tau$ these rectangles are seen as rectangles in $[0,n]\times \Z/n\Z$).
 It follows that
 \[\P(\gamma\notin \widetilde \cP_k(n))\le \P \big( T_k(R_1)=\min_{1\le i\le m_0}T_k(R_i) \big) =\frac 1{m_0}\le C \frac k n.\]
Similarly,
\[\P(\widetilde \gamma\notin \cP_k(n))\le C\frac kn.\]
The conclusion follows from the two previous inequalities together with \eqref{eq:inclusion}.
\end{proof}

\section{The limit shape and the geometry of geodesics}\label{sec 5}

In this section we show that the geometry of geodesics can be controlled using information about the limit shape such as Assumption~\eqref{ass:sides} and Assumption~\eqref{ass:uniform curbature}. 
In particular, we prove that the vertical fluctuations are constrained under the assumption  \eqref{ass:uniform curbature}.
\begin{prop}\label{prop:6.2}
Suppose that $G$ satisfies \eqref{eq:assumption exp1}+\eqref{eq:assumption abs1} or \eqref{eq:assumption atomic} and in addition that \eqref{ass:uniform curbature} holds. Then, for any $\ep>0$ there exists $c>0$ such that
\[\P(T_n=\tau(n, n^{3/4+\ep}))\ge 1-\exp (-cn^{\epsilon }).\]
\end{prop}

 In the following proposition, we bound the number of times a left-right geodesic in a rectangle  intersects a given vertical line.

\begin{prop}\label{cor:number of sides} 
    Let $\epsilon >0$ and suppose that $G$ satisfies \eqref{eq:assumption atomic} and that \eqref{ass:sides} holds for $s$ sufficiently large depending on $\epsilon $. There exist $C_\ep,c_\ep>0$ such that the following holds. Let $R=[0,n]\times [0,k]$ be a rectangle with $4k\le n$. Let $\pi $ be the intersection of all the minimal left-right crossing of $R$. For any $x\in [0,n]$ we have that 
    \begin{equation}
        \mathbb P \big( |\pi \cap \mathcal C_x | \ge k^\epsilon  \big) \le C_\ep\exp (-k^{c_\epsilon }),
    \end{equation}
    where $\mathcal C_x := \{(x,y) : y\in [0,k]\}$ is the $x$ column of $R$.
\end{prop}

\subsection{Proof of Proposition~\ref{prop:6.2}}

For the proof of Proposition~\ref{prop:6.2}, we will need the next theorem which follows from Talagrand's concentration inequality \cite{Talagrand1995} together with a bound on the non-random fluctuations due to Alexander \cite[Theorem~3.2]{alexander}. These results are discussed in \cite[Theorem~2.7 and Theorem~4.5]{dembin2024coalescence}.

\begin{thm}\label{thm:Talagrand}
    Suppose that $G$ satisfies \eqref{eq:assumption abs1}+\eqref{eq:assumption exp1} or \eqref{eq:assumption atomic}. Then, there exists $c>0$ such that for all $x\in \mathbb Z ^2$ and any $ \sqrt{\|x\|} \log ^2 \|x\| \le t \le \|x\|$ we have that 
    \begin{equation}
        \mathbb P \big(  |T(0,x)-\mu (x) | \ge t  \big) \le \exp (-ct^2/\|x\|).
    \end{equation}
\end{thm}
The following lemma is an easy consequence of Theorem \ref{thm:Talagrand}.
\begin{lem}\label{lem:2}
    Suppose that $G$ satisfies \eqref{eq:assumption abs1}+\eqref{eq:assumption exp1} or \eqref{eq:assumption atomic} and in addition that \eqref{ass:uniform curbature} holds. There exists $c>0$ such that the following holds. Let $u=(u_1,u_2),v=(v_1,v_2)$ and $w=(w_1,w_2)$ such that $u_1=0,v_1=n$ and $|u_2-w_2|\ge n^{3/4+\epsilon }$. We also let $u':=(0,\lfloor n/2 \rfloor )$ and $v':=(n,\lfloor n/2 \rfloor )$. We have that
    \begin{equation}
        \mathbb P \big( T(u,w)+T(w,v) \le T(u',v') \big) \le \exp (-cn^{\epsilon }).
    \end{equation}
\end{lem}

\begin{proof}
    Let $w':=(w_1,\lfloor n/2 \rfloor )$. Using \eqref{ass:uniform curbature} 
we have that $\mu (u-w)\ge \mu (u'-w')+ c|u_2-w_2|^2/n$ and using the fact that $\mu $ is symmetric aroung the $x$ axis we have $\mu (w-v) \ge \mu (u'-v')$. Thus, by the 
triangle inequality we have
    \begin{equation}
        \mu (u-w)  +\mu (w-v)  \ge \mu (u'-w')+ c|u_2-w_2|^2/n  +\mu (w'-v') \ge 3n^{1/2+\epsilon } + \mu (u'-v').
    \end{equation}
Hence, in order to violate the inequality inside the probability in Lemma~\ref{lem:2}, one of the passage times $T(u,w)$, $T(w,v)$ or $T(u',v')$ has to deviate from its corresponding norm $\mu(u-w)$, $\mu(w-v)$ or $\mu(u'-v')$ by at least $n^{1/2+\epsilon }$ which has probability at most $\exp (-cn^{\epsilon })$ by Theorem~\ref{thm:Talagrand}.
\end{proof}

\begin{proof}[Proof of Proposition~\ref{prop:6.2}]
Let $\Omega _1$ be the event that for all $u \in \{0\}\times [0,n]$, $v\in \{n\}\times [0,n]$ and $w\in [-n^2,n^2]^2$ with $|u_2-w_2|\ge n^{3/4+\epsilon }$ we have that
\begin{equation}
T(u,w)+T(w,v) > T(u',v').
\end{equation}
By Lemma~\ref{lem:2} and a union bound we have that $\mathbb P (\Omega _1)\ge 1-\exp (-cn^{\epsilon })$.

Let $\gamma '$ be the unrestricted geodesic connecting $u'=(0,\lfloor n/2 \rfloor )$ to $v'=(n,\lfloor n/2 \rfloor )$ and let $\Omega _2$ be the event that $|\gamma '|\le n^2$. We have that $\mathbb P (\Omega _2) \ge 1-e^{-cn}$ (see, e.g., \cite[Claim~2.8]{dembin2024coalescence}). On $\Omega _1\cap \Omega _2$ the geodesic $\gamma ' $ will only cross the left and right boundaries of $[0,n]^2$ and therefore $T_n\le T(u',v')$ (as a sub-path of this geodesic will be a proper left-right crossing of $[0,n]^2$). 

Any left-right crossing $p$ of $[0,n]^2$ with vertical displacement bigger than $2n^{3/4+\epsilon }$ contain points $u,v,w$ as in the event $\Omega _1$. Thus, on the event $\Omega _1\cap \Omega _2$ its weight satisfies $T(p)\ge T(u,w)+T(w,v)>T(u',v')\ge T_n$ and hence such a path cannot be a geodesic.
\end{proof}

\subsection{Proof of Proposition~\ref{cor:number of sides}}
Throughout this section we assume that $G$ satisfies \eqref{eq:assumption atomic}. In this case the geodesic between two points is not necessarily unique. It is convenient to choose one of these geodesics in a deterministic and consistent way. To this end, we consider a total ordering $\preccurlyeq ^*$ on simple paths $p$ connecting $u$ and $v$. Let $p,q$ be two such paths and suppose that $\{e_1,e_2,\dots \}$ and $\{f_1,f_2,\dots \}$ are the sets of edges crossed by $p$ and $q$ respectively. Fix some translation invariant, total order on the edges of $\mathbb Z ^2$ (namely, for two edges $e,f$ and $v\in \mathbb Z^2$ we have $e<f$ iff $e+v<v+f$). Suppose that $e_1<e_2<\cdots $ and $f_1<f_2<\cdots $. We say that $p \preccurlyeq ^* q$ if either $|p|<|q|$ or $|p|=|q|$ and $(e_1,e_2,\dots ) \le (f_1,f_2,\dots )$ in lexicographic order (where in each coordinate we use the above total order on the edges). 
We say that a geodesic $\gamma $ from $u$ to $v$ is minimal$^*$ if it is minimal with respect to $\preccurlyeq ^*$. This way of choosing a specific geodesic is consistent in the sense that if $\gamma (u,v)$ is the minimal$^*$ geodesic connecting $u$ and $v$ and $w,z$ are vertices along this geodesic, then the minimal$^*$ geodesic connecting $w$ to $z$ is precisely the corresponding sub-path of $\gamma (u,v)$. This means that minimal$^*$ geodesics can be ``sandwiched" just like in the situations in which there is a unique geodesic.

To prove Proposition~\ref{cor:number of sides}, we need to rule out geodesics (or minimal$^*$ geodesics) in a rectangle that go a long time in the ``wrong direction". A result of this kind was proved in \cite[Proposition 3.1]{dembin2024coalescence} for unrestricted geodesics. We present below a special case of this proposition in the direction $\pi /4$. 

\begin{prop}\label{prop:limit}
Let $\ep>0$ and assume that \eqref{ass:sides} holds for $s\ge 1$ sufficiently large depending on $\ep$. Let $\gamma $ be the minimal$^*$ geodesic from $(0,0)$ to $(n,n)$. There exist $\theta\in(\pi/4,\pi/2)$ depending only on $G$ such that 
 \begin{equation}
     \mathbb P \Big( (\lfloor R\cos\varphi \rfloor,\lfloor R\sin\varphi \rfloor)\in \Gamma \text{ for some } R\ge n^{\epsilon} \text{ and }  \varphi\in(-\pi,\pi], \,|\varphi|\ge \theta\Big) \le C\exp \big( -n^{c_\epsilon } \big) .
 \end{equation}
\end{prop}

Let us note that \cite[Proposition 3.1]{dembin2024coalescence} is stated a bit differently. Indeed, in \cite{dembin2024coalescence} we worked in the case that $G$ is absolutely continuous and so the geodesic between any two points is unique. However, looking at the proof of \cite[Proposition 3.1]{dembin2024coalescence} we observe that the continuity of $G$ was never used and that any geodesic in the proof can be replaced with minimal$^*$ geodesic.

The next corollary is a straightforward consequence of the previous proposition, using the reflection symmetry of $\mathbb Z ^2$ around the lines $x=y$ and $y=-x+n$ (see Remark~\ref{remark:shit} below).

\begin{cor}\label{cor:controlgeodesic}
Let $\ep>0$ and assume that \eqref{ass:sides} holds for $s\ge 1$ sufficiently large depending on $\ep$. Letting $\gamma $ be the minimal$^*$ geodesic from $(0,0)$ to $(n,n)$, we have 
 \begin{equation}
     \mathbb P \Big( \gamma\setminus \big( B((0,0),2n^\ep)\cup B((n,n),2n^\ep) \big)\not \subset (0,n)^2\Big) \le C\exp \big( -n^{c_\epsilon } \big) .
 \end{equation}
\end{cor}

The following lemma is the main step toward the proof of Proposition~\ref{cor:number of sides}. In the proof of the lemma, we ``sandwich" restricted geodesics between unrestricted ones.

\begin{lem}\label{lem:1}
Let $\ep>0$ and assume that \eqref{ass:sides} holds for $s\ge 1$ sufficiently large depending on $\ep$. Let $R=[0,n]\times [0,k]$ be a rectangle with $4k\le n$. Let $x\in [0,n/2]$, $y\in [k^{\epsilon },k-k^{\epsilon }]$ and $u=(x,y)$. Let $\tilde{\gamma }_u$ be the minimal$^*$ geodesic connecting $u$ to the right boundary of $R$ restricted to stay in $R$. Then
    \begin{equation}
        \mathbb P \big( |\tilde{\gamma }_u \cap \mathcal C_x | \ge k^\epsilon  \big) \le C\exp (-k^{c_\epsilon }).
    \end{equation}
\end{lem}

\begin{figure}[htp]
    \centering
\includegraphics[width=17cm]{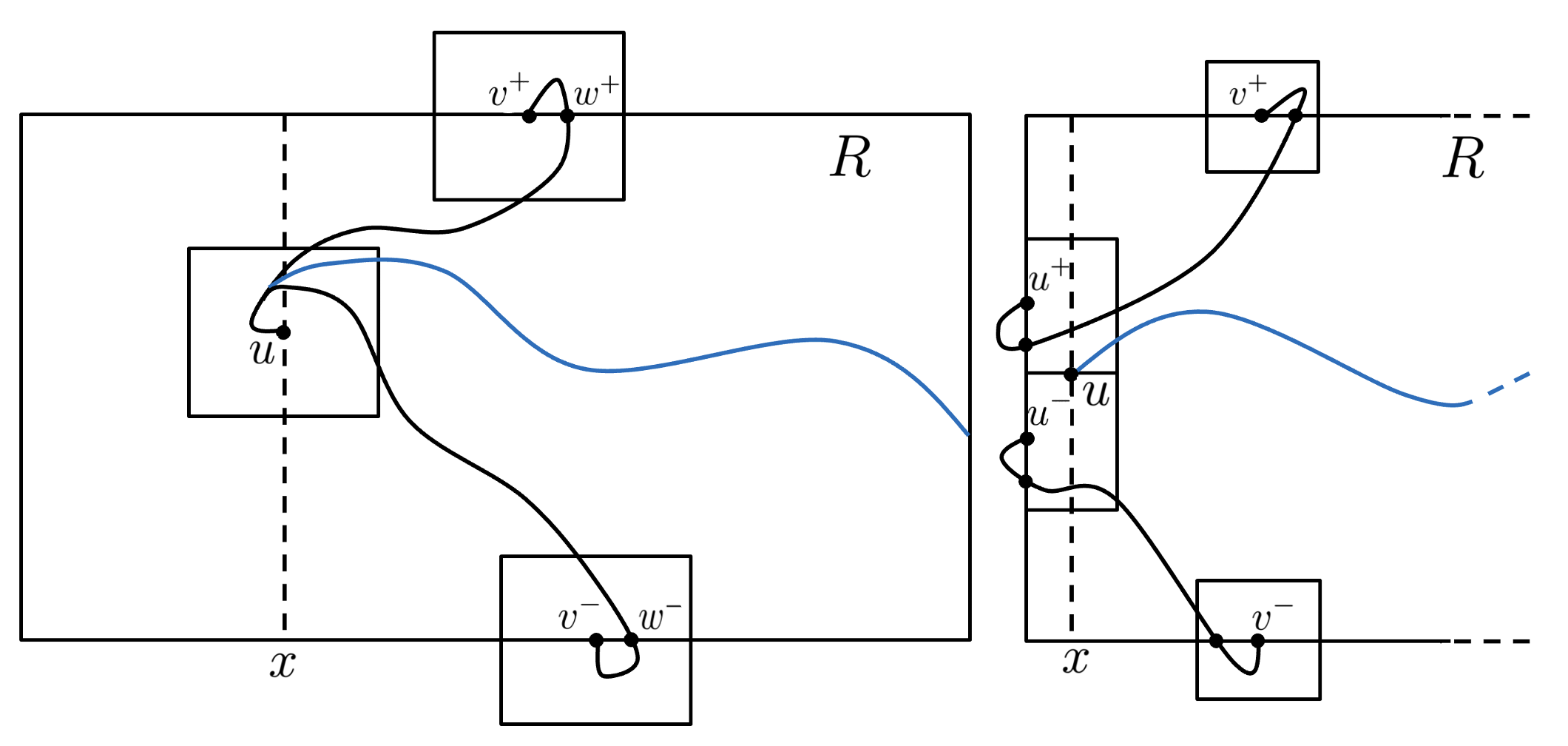}
    \caption{The events $\Omega $ and $\Sigma$. When $x\ge k^{\delta }$, the event $\Omega $ depicted on the left, forces the geodesic $\tilde{\gamma}_u$ (in blue) to intersect $\mathcal C_x$ only inside $B(u,k^\delta )$. When $x< k^{\delta }$, the event $\Sigma $ on the right, forces the geodesic $\tilde{\gamma}_u$ to intersect $\mathcal C_x$ only inside $B(u,2h)$ with $h=\Theta (k^{\delta })$.}
    \label{fig:cyclic}
\end{figure}

\begin{proof}[Proof of Lemma~\ref{lem:1}]
For $v,w\in R$ we let $\tilde{\gamma }(v,w)$ be the minimal$^*$ geodesic from $v$ to $w$ that is restricted to stay in $R$ and let $\gamma (v,w)$ be the unrestricted minimal$^*$ geodesic from $v$ to $w$. Fix $\delta :=\epsilon /2$.

We start with the case $x\ge k^{\delta }$, so that $B(u,k^{\delta }) \subseteq R$. Consider the vertices
    \begin{equation}
        v^-=v^-(u):=(x+y,0) \quad \text{and}\quad v^+=v^+(u):=(x+k-y,k).
    \end{equation}
    These are the two points on the boundary of $R$ that are at $-45$ and $45$ degrees respectively from $u$. Define the event
    \begin{equation}
        \Omega ^-:= \Big\{ \gamma (u,v^-) \setminus \big( B(u,k^\delta ) \cup B(v^-,k^\delta ) \big) \subseteq (x,x+y)\times (0,y) \Big\}.
    \end{equation}
    In words, $\Omega ^-$ is the event that the geodesic $\gamma (u,v^-)$ is contained inside the open square with opposite corners $u$ and $v^-$ except for short parts close to its endpoints (see Figure~\ref{fig:cyclic}). Symmetrically, we define the event 
        \begin{equation}
        \Omega ^+:= \Big\{ \gamma (u,v^+) \setminus \big( B(u,k^\delta ) \cup B(v^+,k^\delta ) \big) \subseteq (x,x+k-y)\times (y,k) \Big\}
    \end{equation}
and let $\Omega :=\Omega ^+ \cap \Omega ^-$. By Corollary~\ref{cor:controlgeodesic} we have that $\mathbb P (\Omega )\ge 1-C\exp (-k^{c_\epsilon })$.

It suffices to prove that on $\Omega $, we have $|\tilde{\gamma } _u\cap \mathcal C _x| \le k^\epsilon $. Let $w^{\pm }$ be the first intersection of $\gamma (u,v^{\pm})$ with the inner boundary of $R$ (coming from $u$). On the event $\Omega $, we have that $\tilde{\gamma }(u,w^\pm )=\gamma (u,w^\pm )$ and $\tilde{\gamma }(u,w^\pm ) \setminus B(u,k^{\epsilon })\subseteq (x,n]\times [0,k]$. Thus, the geodesic $\tilde{\gamma } _u$ is sandwiched between $\tilde{\gamma }(u,w^-)$ and $\tilde{\gamma }(u,w^+)$ and hence $\tilde{\gamma }_u \setminus B(u,k^{\delta })\subseteq (x,n]\times [0,k]$ and $|\tilde{\gamma } _u\cap \mathcal C _x| \le 3k^\delta \le k^{\epsilon }$.

The last proof fails when $x<k^{\delta }$ as the geodesics $\gamma (u,w^{\pm })$ may turn backward from $u$ for a short distance and exit $R$ (so that $\tilde{\gamma }(u,w^\pm )\neq \gamma (u,w^\pm )$). To handle this case, we define slightly different events. Fix $M$ sufficiently large depending only on the edge distribution, let $h:=\lceil Mk^{\delta } \rceil <k^{\epsilon }/2$ and consider the vertices
\begin{equation}
    u^-:=(0,y-h ), \quad  u^+:=(0,y+h ), \quad 
    v^-:=(y-h ,0), \quad  v^+:=(k-y-h,k).
\end{equation}
Define the events
\begin{equation}
    \Sigma ^{\pm} := \Big\{ \gamma (u^{\pm},v^{\pm})\setminus \big( B(u^{\pm}, h )\cup B(v^{\pm },k^{\delta})\big)\subseteq (k^{\delta },k)\times [0,k]  \Big\} \quad \text{and}\quad \Sigma :=\Sigma ^-\cap \Sigma ^+.
\end{equation}
See Figure~\ref{fig:cyclic}. 
Next, note that when $M$ is sufficiently large depending on the angle $\theta <\pi /2$ from Proposition~\ref{prop:limit}, any point $x_0$ in the set $([0,k^\delta]\times [0,k])\setminus B(u^\pm,h)$, satisfies  $x_0=u^\pm+Re^{i\varphi}$ with $|\varphi|\ge \theta$ and $R\ge h$. Thus, for such $M$, as in Corollary~\ref{cor:controlgeodesic}, it follows from Proposition~\ref{prop:limit} that $\mathbb P (\Sigma )\ge 1-C\exp (-k^{c_\epsilon })$.

On  $\Sigma $, there are vertices $w^{\pm},z^{\pm}$ along $\gamma (u^{\pm},v^{\pm})$ on the boundary of $R$ with $w^{\pm}\in B(u^{\pm},h)$ and $z^{\pm}\in B(v^{\pm},k^{\delta})$ for which $\gamma (w^{\pm},z^{\pm})=\tilde{\gamma } (w^{\pm},z^{\pm})$. Moreover, on $\Sigma $, the geodesic $\tilde{\gamma _u}$ is trapped between $\tilde{\gamma } (w^{-},z^{-})$ and $\tilde{\gamma } (w^{+},z^{+})$ and it follows that $\tilde{\gamma}_u\setminus B(u,2h)\subseteq (k^{\delta },n)\times [0,k]$ and $|\tilde{\gamma } _u\cap \mathcal C _x| \le 9h \le k^{\epsilon }$.
\end{proof}

\begin{proof}[Proof of Proposition \ref{cor:number of sides}]
    It suffices to prove the proposition when $\pi $ is replaced by $\gamma$ which is the minimal$^*$ left-right geodesic in $R$. By symmetry we may assume that $x\le n/2\le n-2k$. If $|\gamma \cap \mathcal C _x|\ge k^{\epsilon }$ then there is a vertex $u=(x,y)\in \mathcal C _x$ with $y\in [k^{\epsilon }/3,k-k^{\epsilon }/3]$ for which the sub-geodesic $\tilde{\gamma } _u\subseteq \gamma $ connecting $u$ to the right boundary of $R$ satisfies $|\gamma _u \cap \mathcal C _x| \ge k^{\epsilon }/3$. Thus, the proposition follows from Lemma~\ref{lem:1} (with a slightly smaller $\epsilon >0$) and a union bound over $u=(x,y)$ with $y\in [k^{\epsilon }/3,k-k^{\epsilon }/3]$.
\end{proof}

\begin{remark}\label{remark:shit}
    In the proofs of Proposition \ref{cor:number of sides} and Corollary~\ref{cor:controlgeodesic} we use the reflection symmetry of $\mathbb Z^2$. Let us note the reflection of a minimal$^*$ path is not necessarily minimal$^*$ as the total order on the edges used in the definition of minimal$^*$ is not symmetric. This is of course not an issue since we can prove Lemma~\ref{lem:1} and Proposition~\ref{prop:limit} for minimal$^{**}$ geodesics which are defined in the same way as minimal$^*$ but with a reflected total order on the edges. Then we reflect the environment to obtain the desired estimate for minimal$^*$ geodesics going in the oposite direction.
\end{remark}

\bibliographystyle{plain}
\bibliography{biblio}
\end{document}